\documentclass{amsart}

\usepackage{amssymb}
\usepackage{amstext}
\usepackage{amsmath}
\usepackage{amscd}
\usepackage{amsthm}
\usepackage{amsfonts}
\usepackage{enumerate}
\usepackage{graphicx}
\usepackage{color}
\usepackage{here} 
\usepackage{bm} 
\usepackage{mathrsfs}
\usepackage{mathtools}
\usepackage{latexsym}
\usepackage{booktabs}
\usepackage{fancyhdr}
\usepackage{subcaption}
\usepackage{tikz}
\usepackage{tikz-cd}
\usetikzlibrary{arrows}
\usetikzlibrary{decorations.markings}
\usetikzlibrary{patterns,decorations.pathreplacing}
\usepackage{mathtools}
\usepackage{pdfpages}
\usepackage[section]{algorithm}
\usepackage{algpseudocodex}

\usepackage{hyperref}
\hypersetup{
  colorlinks=true,
  citecolor=red,
  linkcolor=blue
}
\theoremstyle{plain}
\newtheorem{theorem}{Theorem}[section]
\newtheorem{lemma}[theorem]{Lemma}
\newtheorem{proposition}[theorem]{Proposition}
\newtheorem{corollary}[theorem]{Corollary}
\newtheorem{conjecture}[theorem]{Conjecture}
\newtheorem{question}[theorem]{Question}

\theoremstyle{definition}

\newtheorem{definition}[theorem]{Definition}
\newtheorem{example}[theorem]{Example}

\theoremstyle{remark}
\newtheorem{remark}[theorem]{Remark}

\newcommand{\groebner}{Gr\"{o}bner }
\newcommand{\froeberg}{Fr\"{o}berg }
\newcommand{\moreno}{Moreno-Soc\'{i}as }

\newcommand{\intnonneg}{\mathbb{Z}_{\geq 0}}

\newcommand{\calA}{\mathcal{A}}
\newcommand{\calB}{\mathcal{B}}

\newcommand{\calG}{\mathcal{G}}
\newcommand{\calM}{\mathcal{M}}
\newcommand{\Abb}{\mathbb{A}}

\newcommand{\QQ}{\mathbb{Q}}

\newcommand{\kk}{\Bbbk}

\newcommand{\ab}{{\bf a}}
\newcommand{\bb}{{\bf b}}

\newcommand{\tbar}{{\overline{t}}}
\newcommand{\xbar}{{\overline{x}}}

\DeclareMathOperator{\GL}{GL}
\DeclareMathOperator{\borel}{B}

\DeclareMathOperator{\ini}{in}
\DeclareMathOperator{\HS}{HS}
\DeclareMathOperator{\HF}{HF}
\DeclareMathOperator{\maxGBdeg}{maxGBdeg}

\DeclareMathOperator{\LM}{LM}
\DeclareMathOperator{\LC}{LC}
\DeclareMathOperator{\maxdeg}{maxdeg}

\title{Initial ideals of generic ideals and variations of Moreno-Soc\'{i}as conjecture}
\date{}

\author[K. Tani]{Koichiro Tani}
\address{Department of Pure and Applied Mathematics, Graduate School of Information
Science and Technology, Osaka University, Osaka, Japan}
\email{tani-k@ist.osaka-u.ac.jp}

\begin{document}
\begin{abstract}
    It is known that the initial ideals of generic ideals are the same.
    \moreno conjectured that the initial ideal of generic ideals with respect
    to the degree reverse lexicographic order is weakly reverse lexicographic.
    In the first half of this paper, we study the initial ideal of generic ideals
    for arbitrary monomial order and prove that the initial ideal of generic ideals
    is Borel-fixed. It can be considered as a weakened version of
    \moreno conjecture. In the second half, we propose a new method
    of the computation of the initial ideal of generic ideals using
    stability condition of \groebner bases.
    We apply the method in the case of lexicographic order and study
    the relationship between the lexsegment ideal and the initial ideal of
    generic ideals. This study aims to bound the maximal degree of \groebner basis.
    At the last, we propose questions that can be considered as a
    lexicographic analogue of \moreno conjecture.
\end{abstract}
\maketitle

\section{Introduction}\label{section:Introduction}
Let $\kk$ be an infinite field and $S = \kk[\xbar] = \kk[x_1, \ldots, x_n]$ a
$\intnonneg$-graded polynomial ring
with $\deg(x_i)=1(1 \leq i \leq n)$. Throughout this paper, let $I$ be a homogeneous ideal
generated by $f_1, \ldots, f_s$ with $\deg(f_i)=d_i$.
About $\intnonneg$-graded $\kk$-algebra $S/I = \bigoplus_{d \in \intnonneg}(S/I)_d$, the
\textit{Hilbert function} and the \textit{Hilbert series} of $S/I$, are defined to be
\begin{equation*}
    \HF(S/I, d) = \dim_{\kk}(S/I)_d, \text{ and } \HS(S/I; t) = \sum_{i \in \intnonneg} \dim_{\kk}(S/I)_i t^i \text{, respectively}.
\end{equation*}

For polynomials in $\kk[\xbar]$
\begin{align*}
    f_1 &= a_{1,1}x_1^{d_1} + a_{1,2}x_1^{d_1-1}x_2 + \cdots + a_{1,r_1}x_n^{d_1} \\
    &\hspace*{80pt} \vdots \\
    f_s &= a_{s,1}x_1^{d_s} + a_{s,2}x_1^{d_s-1}x_2 + \cdots + a_{s,r_s}x_n^{d_s}
\end{align*}
with $a_{j,k} \in \kk$,
each $f_i$ is a $\kk$-linear combination of all monomials
of degree $d_i$.
Let $r_i = \binom{n+d_i-1}{d_i}$, i.e., the number of monomials of
degree $d_i$, and $N = \sum_{i=1}^{s}r_i$.
We assume that the points $\ab \in \Abb_{\kk}^N$ correspond to
ideals $I = (f_1, \ldots, f_s)$.
In other words, for a point $\ab \in \Abb_{\kk}^N$, let $\sigma_{\ab}$
be the canonical specialization homomorphism defined by
\begin{equation*}
    \sigma_{\ab} : \kk[\tbar][\xbar] \to \kk[\xbar], \quad f(\tbar, \xbar) \mapsto f(\ab, \xbar),
\end{equation*}
where $\tbar =\{t_{j,k} \mid 1 \leq j \leq s, 1 \leq k \leq r_j\}$
is a set of parameter variables.
We apply the specialization homomorphism $\sigma_{\ab}$ to polynomials
\begin{align*}
    F_1 &= t_{1,1}x_1^{d_1} + t_{1,2}x_1^{d_1-1}x_2 + \cdots + t_{1,r_1}x_n^{d_1} \\
    &\hspace*{80pt} \vdots \\
    F_s &= t_{s,1}x_1^{d_s} + t_{s,2}x_1^{d_s-1}x_2 + \cdots + t_{s,r_s}x_n^{d_s}
\end{align*}
in $\kk[\tbar][\xbar]$. By this operation,
a point $\ab \in \Abb_{\kk}^N$ corresponds to an ideal $I$ generated by
$\{f_i = \sigma_{\ab}(F_i) \mid i = 1, \ldots, s\}$.
Throughout this paper, we use the notations and settings described above.

We say ideals are \textit{generic} if they satisfy a property on a
Zariski open dense set in $\Abb_{\kk}^N$. We call ideals in such a Zariski open
dense set \textit{generic ideals}. For example, the following holds.

\begin{lemma}[{\cite[Theorem 1]{FroebergGenericHilbertSeries}}]\label{lemma:generic hilbert series}
    For fixed $n,s,d_1, \ldots, d_s$,
    there exists a Zariski open dense set $U \subset \Abb_{\kk}^N$
    such that for all $\ab \in U$, the Hilbert series $\HS(S/I)$ is the same,
    where $I$ corresponds to $\ab$. 
\end{lemma}

The following longstanding conjecture is known as \textit{\froeberg conjecture}.

\begin{conjecture}[{\cite{FroebergConjecture}}]\label{conjecture:Froeberg conjecture}
    Fix $n, s, d_1, \ldots, d_s$. For a generic ideal $I$ described in Lemma~\ref{lemma:generic hilbert series},
    the Hilbert series $\HS(S/I;t)$ is $| \frac{\prod_{i=1}^{s}(1-t^{d_i})}{(1-t)^n} |$.
\end{conjecture}

A formal power series $| \sum_{i=0}^{\infty}a_it^i |$ is defined as $\sum_{i=0}^{\infty}b_it^i$, where $b_i = a_i$
if $a_j > 0$ for all $j \leq i$ and $b_i = 0$ otherwise.
\froeberg conjecture is true for $s \leq n$ because, in this case, generic ideals are
generated by regular sequences (see~\cite[Theorem 1]{FroebergLoefwall}).
About other proved cases, see, for example,~\cite{provedFroebergConjecture1},
~\cite{provedFroebergConjecture3}, and~\cite{provedFroebergConjecture2}.

Let $\calM(\xbar)$ be the set of all monomials consisting of variables $x_1, \ldots, x_n$
and let $\preceq$ be a monomial order on $\calM(\xbar)$.
We suppose $x_1 \succeq \ldots \succeq x_n$ throughout this paper.
For $f \in S$, we define $\ini_{\preceq}(f)$ as the initial term of $f$.
A $\kk$-vector space spanned by $\{\ini_{\preceq}(f) \mid f \in I\}$
is the \textit{initial ideal} of $I$ and write $\ini_{\preceq}(I)$.
A finite set of polynomials $\calG = \{g_1, \ldots, g_t\}$ is a \textit{\groebner basis} with respect to $\preceq$
if $\ini_{\preceq}(I) = (\ini_{\preceq}(g_1), \ldots, \ini_{\preceq}(g_t))$.
A \groebner basis $\{g_1, \ldots, g_t\}$ is \textit{reduced} if they are monic and for each
$g_i$, there is no monomial in $g_i$ divisible by $\ini_{\preceq}(g_j)$ for any $j \neq i$.
It is known that a unique reduced \groebner basis exists for each monomial order and ideal.
We can compute the minimal generators of $\ini_{\preceq}(I)$ by computing reduced \groebner basis of $I$.
For more details on \groebner basis, see, for example,~\cite{BeckerWeispfenning} and~\cite{CoxLittleOSheaIVA}.

The initial ideal with respect to a fixed monomial order also remains the same for generic ideals.
\begin{lemma}[See, e.g.,~{\cite{PardueGenericSequence}}]\label{lemma:initial ideal of generic ideal}
    For a fixed monomial order $\preceq$, 
    there exists a Zariski open dense set $V$ such that for all $\ab \in V$,
    the initial ideal $\ini_{\preceq}(I)$ is the same.
\end{lemma}

A monomial ideal $J$ is \textit{weakly reverse lexicographic} if, for all monomials $m$ in the minimal generators of $J$,
any monomial $m'$ satisfying $\deg(m) = \deg(m')$ and $m \preceq m'$ with respect to the reverse lexicographic order
also belongs to $J$. The following conjecture is known as \textit{\moreno conjecture}.

\begin{conjecture}[{\cite{MorenoSociasConjecture}}]\label{conjecture:Moreno-Socias conjecture}
    Fix $n, s, d_1, \ldots, d_s$. For a generic ideal $I$ described in Lemma~\ref{lemma:initial ideal of generic ideal},
    the initial ideal of $I$ with respect to the degree reverse lexicographic order
    is weakly reverse lexicographic.
\end{conjecture}

The proved cases of \moreno conjecture are documented in, for example,~\cite{CapaverdeGaoMorenoSociasConjecture}.
Pardue~\cite{PardueGenericSequence} showed that \moreno conjecture 
implies \froeberg conjecture.

In this paper, we reconstruct these discussions by using exterior algebra, inspired by~\cite{BayerStillman}.
Thanks to this reconstruction,
we can prove that the initial ideal of generic ideals is Borel-fixed. This can be considered as
a weakened version of \moreno conjecture. For more details,
see Section~\ref{section:The initial ideal of generic ideals is Borel-fixed}.

Furthermore, we propose a new method of computing the initial ideal of generic ideals using
the stability condition of \groebner bases~\cite{KalkbrenerStability}.
Using this method, we study
when the initial ideal of generic ideal with respect to the lexicographic order
becomes lexsegment. This question aims to identify
the worst-case complexity of computing \groebner bases.
\groebner basis is known as a good tool not only for purely mathematical problems but
also for applied mathematical problems across various fields.
However, it is often the case that polynomials of high degree
appear in a \groebner basis and a computation of a \groebner basis takes tremendously long time.
Therefore, investigation of maximal degree of \groebner basis is important for the analysis of
the complexity of the computation of \groebner basis.

Bounding a maximal degree of \groebner basis for a homogeneous ideal is well studied.
For example,~\cite{DubeMcaulayConstant},~\cite{HashemiMacaulayConstant} and~\cite{MayrRitscherMacaulayConstant}
bounded it by using and computing Macaulay constants. Under the fixed Hilbert series, we can bound
the maximal degree of polynomials in the reduced \groebner basis by using the lexsegment ideal
(see Definition~\ref{definition:lexsegment ideal})
and Macaulay's Theorem (see Theorem~\ref{theorem:Macaulay's Theorem}).
This bound is described in Corollary~\ref{corollary:maxGBdeg bound with lexsegment ideal}.
Furthermore,~\cite{BayerThesis} and~\cite{RegularityofLexsegment} computed
the maximal degree of minimal generators of lexsegment ideals.
Finding lexsegment ideals as initial ideals signifies the optimality of the bound.
Using the method we propose, we can identify cases where lexsegment ideals appear as
initial ideals and cases where lexsegment ideals never appear as initial ideals
(see Corollary~\ref{corollary:main corollary2} and Corollary~\ref{corollary:main corollary3}).

At the last, we propose questions
that can be considered as the lexicographic analogue of \moreno conjecture.
For more details, see Section~\ref{section:the initial ideal of generic ideals with respect to lexicographic order}.

\subsection*{Acknowledgement}
The author would like to thank his supervisor Akihiro Higashitani  for his very helpful
comments and instructive discussions. 
This work was supported by the Research Institute for Mathematical Sciences,
an International Joint Usage/Research Center located in Kyoto University.

\section{The initial ideal of generic ideals is Borel-fixed}\label{section:The initial ideal of generic ideals is Borel-fixed}
In this section we reconstruct discussions about generic ideals by using exterior algebra.
This method is analogous to the approach used to prove that the generic initial ideal is Borel-fixed.
The existence of generic initial ideals and Borel-fixedness in characteristic $0$ were proved by Galligo~\cite{Galligo}, and later by
Bayer and Stillman~\cite{BayerStillman} for any order and any characteristic.
Proofs of the existence and Borel-fixedness are written, for example, in~\cite[Section 15.9]{EisenbudText}.

We consider $I_d$ as a $\kk$-subspace of $S_d$ for each degree $d \geq 0$. Let $t = \dim_{\kk}(S/I)_d$.
Then, $\bigwedge^t S_d$ is a $\kk$-vector space with a basis of
\begin{equation*}
    \{m_1 \wedge \ldots \wedge m_t \mid m_1 \succeq \ldots \succeq m_t \text{ are monomials in } S_d\}
\end{equation*}
for a fixed monomial order $\preceq$. We will say that $m = m_1 \wedge \ldots \wedge m_t$ is a
\textit{standard exterior monomial} if $m_1 \succeq \ldots \succeq m_t$.
We order standard exterior monomials lexicographically. That is,
$m = m_1 \wedge \ldots \wedge m_t \geq m' = m_1' \wedge \ldots \wedge m_t'$ if and only if
$m_i \succeq m_i'$ for the smallest $i$ such that $m_i \neq m_i'$.
If $h_1, \ldots, h_t$ are a basis of $I_d$, $\bigwedge^t I_d$ is a $1$-dimensional $\kk$-subspace of $\bigwedge^t S_d$ generated by
$h = h_1 \wedge \ldots \wedge h_t$.
We can expand $h$ and rewrite it as a linear combination of standard exterior monomials, allowing us to determine the largest
standard exterior monomial of $h$ with nonzero coefficients. We also refer to the product of this coefficient and the largest
standard exterior monomial of $h$ as the \textit{initial term} of $h$, denoted by $\ini(h)$. The standard exterior monomial
of $\ini(h)$ consists of monomials that form a basis of $\ini_{\preceq}(I)_d$. Using these notions, we reconstruct discussions
about generic ideals.

\begin{theorem}[{\cite[Theorem 1]{FroebergGenericHilbertSeries}}]\label{theorem:Froeberg generic hilbert series}
    Let $\kk$ be an infinite field and fix positive integers $n, s, d_1, \ldots, d_s$.
    We correspond a point $\ab \in \Abb_{\kk}^N$ to an ideal $I = (f_1, \ldots, f_s)$. Then, there is a Zariski open dense set $U$
    such that, for all $\ab \in U$, the Hilbert series $\HS(S/I)$ is the same.
\end{theorem}
\begin{proof}
    Let
    \begin{align*}
        F_1 &= t_{1,1}x_1^{d_1} + t_{1,2}x_1^{d_1-1}x_2 + \cdots + t_{1,r_1}x_n^{d_1} \\
        &\vdots \\
        F_s &= t_{s,1}x_1^{d_s} + t_{s,2}x_1^{d_s-1}x_2 + \cdots + t_{s,r_s}x_n^{d_s}
    \end{align*}
    be polynomials in $\kk[\tbar][\xbar]$. For each degree $d$,
    we want to choose $\ab$ such that $\dim_{\kk} I_d$ is as large as possible.
    For any $\ab \in \Abb_{\kk}^N$ and $d$, $I_d$ is generated by
    \begin{equation*}
        \calA_d = \bigcup_{i=1}^{s}\{\tau f_i \mid \tau \text{ are monomials of degree } d -\deg(f_i) \geq 0\}
    \end{equation*}
    as a $\kk$-vector space. If there exists a point $\ab$ such that $\dim_{\kk}(I_d) = t$ for an integer $t$, then
    there exist monomials $\tau_{1,1}, \ldots, \tau_{1, l_1}, \ldots, \tau_{s, 1}, \ldots, \tau_{s, l_s}$ such that
   $\sum_{j=1}^{s}l_j = t$ and $\{\tau_{1,1}f_1, \ldots, \tau_{1,l_1}f_1, \ldots, \tau_{s,1}f_s, \ldots, \tau_{s,l_s}f_s\}$
   is a basis of $I_d$. Therefore,
   \begin{equation*}
    f = \tau_{1,1}f_1 \wedge \ldots \wedge \tau_{1,l_1}f_1 \wedge \ldots \wedge \tau_{s,1}f_s \wedge \ldots \wedge \tau_{s,l_s}f_s
   \end{equation*}
   is a basis of $\bigwedge^tI_d$. Since there exists a point $\ab$ such that $f$ is a basis of $\bigwedge^tI_d$,
   \begin{equation*}
    F = \tau_{1,1}F_1 \wedge \ldots \wedge \tau_{1,l_1}F_1 \wedge \ldots \wedge \tau_{s,1}F_s \wedge \ldots \wedge \tau_{s,l_s}F_s
   \end{equation*}
   is not zero, i.e., $F$ contains a standard exterior monomial with a nonzero coefficient that is in $\kk[\tbar]$
   when we expand $F$ and rewrite as a linear combination of standard exterior monomials.
   Let $t_d$ be the largest integer for each degree $d$ such that there exist monomials $\tau_{i,j}$ where $\sum_{j=1}^{s}l_j = t_d$ and
   $F$ is nonzero. Let $\calB_d$ be a set of such $F$'s.
   Also, let $U_d$ be the set of points in $\Abb_{\kk}^N$ that do not vanish on some
   coefficients of some $F \in \calB_d$.
   
   Finally, we prove $U = \bigcap_{d=0}^{\infty}U_d$ is a Zariski open dense set
   of $\Abb_{\kk}^N$. It is sufficient to prove that $U$ is the intersection of
   finitely many $U_d$'s because each $U_d$ is a Zariski open dense set of
   $\Abb_{\kk}^N$.
   The proof of finiteness is written in the proof of~\cite[Theorem 1]{FroebergGenericHilbertSeries}.
   The finiteness can be proved by the fact that only finitely
   many formal power series can appear as a Hilbert series $\HS(S/I; t)$ when $I$
   corresponds to $\ab \in \Abb_{\kk}^N$.
   
   Thus, there is a Zariski open dense set $U$ such that for all $\ab \in U$, the Hilbert series $\HS(S/I)$ is the same.
   In particular, if $I$ corresponds to some $\ab \in U$ and $I'$ corresponds to any $\ab \in \Abb_{\kk}^N$, then
   $\dim_{\kk}(S/I)_d \leq \dim_{\kk}(S/I')_d$ for all $d \geq 0$.
\end{proof}

\begin{theorem}[{\cite{PardueGenericSequence}}]\label{theorem:existence of initial ideal of generic case}
    We use the same notation as in Theorem~\ref{theorem:Froeberg generic hilbert series}.
    We fix positive integers $n, s, d_1, \ldots, d_s$ and a monomial order $\preceq$.
    Then, there is a Zariski open dense set $V \subset U$
    such that, for all $\ab \in V$, the initial ideal $\ini_{\preceq}(I)$ is the same.
\end{theorem}
\begin{proof}
    We use the same notation as in the proof of Theorem~\ref{theorem:Froeberg generic hilbert series}.
    For each degree $d$, we can take the largest standard exterior
    monomial $m^{(d)}$ from the set of standard exterior monomials contained in some $F \in \calB_d$.
    In other words, for each $F \in \calB_d$, $\ini(F)$ is the product of a coefficient in $\kk[\tbar]$
    and a standard exterior monomial, which is the largest one in $F$. We can take $m^{(d)}$ as
    the largest standard exterior monomial from $\{\ini(F) \mid F \in \calB_d\}$.
    If $\ab \in U_d$, at least one of $F \in \calB_d$ becomes a basis of $\bigwedge^{t_d}I_d$.
    Then, one of standard exterior monomials of $F$ becomes a basis of $\bigwedge^{t_d}\ini_{\preceq}(I)_d$.
    Therefore, we have the following equation:
    \begin{equation}\label{equation:the largest standar exterior monomial}
    m^{(d)} = \max_{\ab \in U_d}
    \{m = m_1 \wedge \ldots \wedge m_{t_d} \mid m \text{ is a basis of } \bigwedge^{t_d}\ini_{\preceq}(I)_d\}.
    \tag{$\star$}
    \end{equation}
    Let $V_d$ be the set of points in $U_d$ that do not vanish on the coefficient of some $\ini(F)$
    whose standard exterior monomial is $m^{(d)}$.
   Let $J_d$ be a $\kk$-subspace of $S_d$ spanned by monomials forming $m_d^{(d)}$.
   
   We next show that $J = \bigoplus_{d=0}^{\infty}J_d$ is an ideal of $S$. It is sufficient
   to show $S_1J_d \subset J_{d+1}$ for each $d$. Since $V_d$ and $V_{d+1}$ are Zariski open
   dense sets, the intersection $V_d \cap V_{d+1}$ is not empty and there exists a point
   $\ab \in V_d \cap V_{d+1}$. If $I$ corresponds to $\ab$,
   we have $\ini_{\preceq}(I)_d = J_d$ and $\ini_{\preceq}(I)_{d+1} = J_{d+1}$,
   therefore the assertion is true.
   
   Finally, we prove $V = \bigcap_{d=0}^{\infty}V_d$ is a Zariski open dense set of $U_d$.
   As well the proof of Theorem~\ref{theorem:Froeberg generic hilbert series},
   it is sufficient to prove that $V$ is the intersection of finitely many $V_d$'s.
   The ideal $J$ is finitely generated. By the definition of $V_d$, we have $V_d \subset U_d$.
   Also, $U$ is the intersection of finitely many $U_d$'s.
   Therefore, there exists $e \geq 0$ such that $J$ is generated by at most degree $e$ monomials
   and $\bigcap_{d=0}^{e}V_d \subset \bigcap_{d=0}^{e}U_d = U$.

   We prove $V = \bigcap_{d=0}^{e}V_d$.
   If $\ab \in \bigcap_{d=0}^{e}V_d$ and $I$ corresponds to $\ab$, we have
   $\ini_{\preceq}(I)_d = J_d$ for all $d \leq e$. Therefore, $\ini_{\preceq}(I) \supset J$.
   Since $\ab \in U$, we have $\ini_{\preceq}(I)_d = J_d = t_d$ for all $d \in \intnonneg$.
   Thus, $\ini_{\preceq}(I) = J$.
\end{proof}

\begin{remark}
    By the definition of $U$ and $V$, these two Zariski open dense sets are not the
    same in general. As Example~\ref{example:U not equal to V}, there are two
    ideals $I$ and $J$ such that the two points corresponding to $I$ and $J$
    are in $U$ but the initial ideals of $I$ and $J$
    are not the same. To avoid confusion, we refer to the ideals in $U$ as
    \textit{$U$-generic ideals} and those in $V$ as \textit{$V$-generic ideals}.
    \froeberg conjecture is a conjecture about $U$-generic ideals and
    \moreno conjecture is a conjecture about $V$-generic ideals.
\end{remark}

\begin{example}\label{example:U not equal to V}
    Let $\kk = \QQ, n = s = 3$ and $d_1 = d_2 = d_3 = 2$.
    Let $I, J$ be the ideals in $\kk[x_1, x_2, x_3]$ defined by
    \begin{align*}
        I = (&x_1^2 + x_1x_3 + x_2x_3 + x_3^2, \\ 
        &x_1^2 + x_1x_2 + x_1x_3 + x_3^2, \\
        &x_1^2 + x_1x_2 - x_1x_3 + x_2^2 - x_2x_3 - x_3^2), \\
        J = (&x_1^2 + x_1x_3 + x_2^2 + x_2x_3 + x_3^2, \\ 
        &x_1x_2 + x_1x_3 - x_2^2 + x_2x_3 + x_3^2, \\
        &x_1^2 + x_1x_2 + x_1x_3 + x_2x_3 + x_3^2),
    \end{align*}
    respectively. Then, the Hilbert series of $I$ and $J$ is the same $t^3 + 3t^2 + 3t + 1$.
    Therefore, both of two ideals are $U$-generic because
    a $U$-generic ideal is generated by regular sequence in the case $s \leq n$.
    However, initial ideals with respect to the degree reverse lexicographic order $\preceq$
    are not the same as shown below:
    \begin{align*}
        \ini_{\preceq}(I) &= (x_1^2, x_1x_2, x_2^2, x_1x_3^2, x_2x_3^2, x_3^4), \\
        \ini_{\preceq}(J) &= (x_1^2, x_1x_2, x_1x_3, x_2^3, x_2^2x_3, x_2x_3^2, x_3^4).
    \end{align*}
    Therefore, one of $I$ and $J$ is not a $V$-generic ideal but both of
    $I$ and $J$ are $U$-generic ideals.
\end{example}

We denote $\GL_n(\kk)$ as the general linear group, which is the group of all invertible
$n \times n$-matrices over $\kk$. Also, we denote $\borel_n(\kk)$ as the Borel subgroup,
which is the subgroup of $\GL_n(\kk)$ consisting of all invertible upper triangular matrices.
The general linear group acts on the polynomial ring as follows.
For $\alpha = (\alpha_{i,j}) \in \GL_n(\kk)$ and $f(x_1, \ldots, x_n) \in S$,
let $\alpha$ act on $f$ by $\alpha \cdot f = f(\alpha \cdot x_1, \ldots, \alpha \cdot x_n)$,
where $\alpha \cdot x_j = \sum_{i=1}^{n}\alpha_{i,j}x_i$. Let
$\alpha \cdot I = \{\alpha \cdot f \mid f \in I\}$ for
a matrix $\alpha \in \GL_n(\kk)$ and an ideal $I \subset S$. We say that $I$ is \textit{Borel-fixed}
if $\alpha \cdot I = I$ for all $\alpha \in \borel_n(\kk)$.

\begin{theorem}\label{theorem:main theorem1}
    We fix positive integers $n, s, d_1, \ldots, d_s$ and a monomial order $\preceq$.
    Then, the initial ideal of $V$-generic ideals is Borel-fixed.
\end{theorem}
\begin{proof}
    We use the same notation as in the proof of Theorem~\ref{theorem:Froeberg generic hilbert series}
    and Theorem~\ref{theorem:existence of initial ideal of generic case}.
    Let $I$ be a $V$-generic ideal, and $J = \ini_{\preceq}(I)$.
    For $i, j$ with $i < j$, let $\gamma_{i,j}^c = 1 + \gamma$  where $1$ is the identity matrix
    and $\gamma$ is a strictly upper triangular matrix with exactly one nonzero entry, $(i,j)$,
    equal to $c$. Since $\borel_n(\kk)$ is generated by such matrices and diagonal matrices,
    it is sufficient to show that $J$ is fixed under the action of these matrices.
    We know that $J$ is fixed under the action of diagonal matrices because $J$ is a monomial
    ideal (see Proposition~\ref{prop:characterization of Borel-fixed}).
    
    We prove $\gamma_{i,j}^c \cdot J = J$ for any $\gamma_{i,j}^c$. It is sufficient to show
    $\gamma_{i,j}^c \cdot J_d = J_d$ for all $d$.
    Let $h = h_1 \wedge \ldots \wedge h_{t_d}$ be a basis of $\bigwedge^{t_d}I_d$
    and assume $m^{(d)} = \ini_{\preceq}(h_1) \wedge \ldots \wedge \ini_{\preceq}(h_{t_d})$,
    that is a basis of $\bigwedge^{t_d}J_d$. Since $J$ is the initial ideal of $V$-generic ideals,
    $m^{(d)}$ satisfies the equation~(\ref{equation:the largest standar exterior monomial})
    for all $d$.
    Then,
    \begin{equation*}
        \gamma_{i,j}^c \cdot m^{(d)} = \gamma_{i,j}^c \cdot m_1^{(d)} \wedge \ldots \wedge \gamma_{i,j}^c \cdot m_{t_d}^{(d)}
    \end{equation*}
    is a basis of $\bigwedge^{t_d}\gamma_{i,j}^c \cdot J_d$.
    We assume $\gamma_{i,j}^c \cdot m^{(d)} \neq m^{(d)}$.
    There exists a term in $\gamma_{i,j}^c \cdot m^{(d)}$ that has nonzero coefficient
    and larger than $m^{(d)}$. Since $\gamma_{i,j}^c$ is an upper triangular matrix, all terms
    in $\gamma_{i,j}^c \cdot m^{(d)}$ except for $m^{(d)}$ are strictly larger than $m^{(d)}$.
    Let $m$ be one of the standard exterior monomials in $\gamma_{i,j}^c \cdot m^{(d)}$ with nonzero
    coefficient. Let $a$ be the coefficient of $m$ in $\gamma_{i,j}^c \cdot m^{(d)}$.
    If there exists a diagonal matrix $\delta \in \GL_n(\kk)$ such that
    $\gamma_{i,j}^c \delta \cdot h$ has standard exterior monomial $m$ with nonzero coefficient,
    this contradicts the equation~(\ref{equation:the largest standar exterior monomial}).
    The reason is as follows: if $I$ is generated by homogeneous polynomials $f_1, \ldots, f_s$
    with $\deg(f_i) = d_i$, $\gamma_{i,j}^c \delta \cdot I$ is generated by homogeneous
    polynomials $\gamma_{i,j}^c \delta \cdot f_1, \ldots, \gamma_{i,j}^c \delta \cdot f_s$ with
    $\deg(\gamma_{i,j}^c \delta \cdot f_i) = d_i$. Therefore, $\gamma_{i,j}^c \delta \cdot I$
    corresponds to another point $\bb \in \Abb_{\kk}^N$. Moreover, we know
    $\dim_{\kk}(I_d) = \dim_{\kk}(\gamma_{i,j}^c \delta \cdot I_d)$ because
    $\gamma_{i,j}^c \delta \in \GL_n(\kk)$. Thus, we have $\bb \in U \subset U_d$ for all $d$.

    Finally, we have to show the existence of a diagonal matrix $\delta$.
    The proof of this is the same as the proof of~\cite[Theorem 15.20]{EisenbudText}.
    Thus, we have $\gamma_{i,j}^c \cdot J = J$ for all $d$ and $J$ is Borel-fixed.
\end{proof}

The combinatorial characterization of Borel-fixed monomial ideals is well known.
For $s, t \in \intnonneg$, we say that $s \preceq_p t$ if
$\binom{t}{s} \not\equiv 0 \pmod p$.
If $p = 0$, then $\preceq_p$ is the usual total order $\leq$.

\begin{proposition}[{\cite[Theorem 15.23]{EisenbudText}}]\label{prop:characterization of Borel-fixed}
    Let $J \subset S$ be an ideal, and let $p \geq 0$ be the characteristic of $\kk$.
    \begin{enumerate}
        \item $J$ is fixed under the action of diagonal matrices if and only if $J$ is generated
        by monomials.
        \item $J$ is fixed under the action of $\borel_n(\kk)$ if and only if $J$ is generated by
        monomials and the following condition is satisfied for all $i < j$ and all monomial
        generators $m$ of $J$:
        If $m$ is divisible by $x_j^t$ but not by $x_j^{t+1}$, then $(x_i / x_j)^s m \in J$
        for all $i < j$ and $s \preceq_p t$.
    \end{enumerate}
\end{proposition}

If a monomial ideal $J$ is weakly reverse lexicographic, then $J$ is Borel-fixed because
$(x_i / x_j)^s m$ is larger than $m$ with respect to the degree reverse lexicographic order,
under the assumption in (2) of Proposition~\ref{prop:characterization of Borel-fixed}.
Therefore, Theorem~\ref{theorem:main theorem1} does not contradict \moreno conjecture.
Theorem~\ref{theorem:main theorem1} can be considered as a weakened version of
\moreno conjecture, generalized to arbitrary monomial order.

\section{the initial ideal of generic ideals with respect to the lexicographic order}\label{section:the initial ideal of generic ideals with respect to lexicographic order}
In this section, we consider the initial ideal of $U$-generic ideals with respect to the lexicographic order
and the relationship between them and lexsegment ideals.
We introduce the definition of lexsegment ideals
and Macaulay's Theorem.
For more details on lexsegment ideals and Macaulay's Theorem, see~\cite[Section 6.3]{MonomialIdeals}
or~\cite[Section 2.4]{MillerSturmfels}.

\begin{definition}\label{definition:lexsegment ideal}
    Let the monomial order $\preceq$ be the lexicographic order.
    A monomial ideal $J$ is called the \textit{lexsegment ideal} if,
    for all monomials $m \in J$, any monomial $m'$ satisfying
    $\deg(m) = \deg(m')$ and $m \preceq m'$ also belongs to $J$.
\end{definition}

\begin{theorem}[Macaulay's Theorem]\label{theorem:Macaulay's Theorem}
    For any homogeneous ideal $I$, there uniquely exists a lexsegment ideal $J$ such that $\HS(S/I; t) = \HS(S/J; t)$.
    Moreover, for any $d \geq 0$, $J$ has at least as many generators in degree $d$ as any other monomial
    ideal with the same Hilbert series.
\end{theorem}

\begin{definition}
    Given a Hilbert series $\HS(S/I; t)$ for a homogeneous ideal $I$,
    we define the lexsegment ideal $J$ induced by Macaulay's Theorem
    (i.e., satisfying $\HS(S/J; t) = \HS(S/I; t)$)
    as \textit{the lexsegment ideal of $\HS(S/I; t)$}.
\end{definition}

We define $\maxGBdeg_{\preceq}(I)$ as the maximal degree of polynomials in the reduced \groebner basis
and $\maxdeg(I)$ as the maximal degree of minimal generators of $I$.
We can bound $\maxGBdeg_{\preceq}(I)$ with the lexsegment ideals of the same Hilbert series
$\HS(S/I; t)$.

\begin{corollary}\label{corollary:maxGBdeg bound with lexsegment ideal}
    For any homogeneous ideal $I$, let $J$ be the lexsegment ideal of
    $\HS(S/I; t)$. Then, we have $\maxGBdeg_{\preceq}(I) \leq \maxdeg(J)$ for any monomial order $\preceq$.
\end{corollary}
\begin{proof}
    Fix a monomial order $\preceq$. By the definition of reduced \groebner basis,
    we have $\maxGBdeg_{\preceq}(I) = \maxdeg(\ini_{\preceq}(I))$.
    It is well known that $\HS(S/I; t) = \HS(S/\ini_{\preceq}(I); t)$ for
    any homogeneous ideal $I$ (see, e.g.,~\cite[Theorem 15.26]{EisenbudText}).
    Therefore, by Macaulay's Theorem, we have
    \begin{equation*}
        \maxGBdeg_{\preceq}(I) = \maxdeg(\ini_{\preceq}(I)) \leq \maxdeg(J).
    \end{equation*}
\end{proof}

We check whether the bound in Corollary~\ref{corollary:maxGBdeg bound with lexsegment ideal}
is optimal. If there exist an ideal $I$ and a monomial order $\preceq$ such that $\ini_{\preceq}(I)$
is the lexsegment ideal of $\HS(S/I; t)$, then the equality of the inequality in Corollary~\ref{corollary:maxGBdeg bound with lexsegment ideal}
holds. Thus, we want to find such an ideal.
We fix the lexsegment ideal and compare homogeneous ideals $I$ under
the fixed Hilbert series. For simplicity, we assume $I$ is a $U$-generic ideal.
Because of Theorem~\ref{theorem:main theorem2}, we can either find ideals $I$ whose
initial ideals are lexsegment for the generic case, or determine that such ideals
do not exist at all.

\begin{theorem}\label{theorem:main theorem2}
    Fix $n, s, d_1, \ldots, d_s$. Let the monomial order be the lexicographic order.
    If there exists $\ab \in U$ such that the initial ideal of $I$ is lexsegment,
    then the initial ideal of $V$-generic ideals is lexsegment.
\end{theorem}
\begin{proof}
    Since the monomial order $\preceq$ is lexicographic order, the largest
    standard exterior monomial in $m_{\mathrm{lex}}^{(d)} \in \bigwedge^{t_d}S_d$ consists of
    the $t_d$ largest monomials in $S_d$ with respect to the lexicographic order
    for each $d$. Let $J$ be the initial ideal of $V$-generic ideals and
    $m^{(d)}$ a basis of $\bigwedge^{t_d}J_d$.
    By the equation~(\ref{equation:the largest standar exterior monomial}),
    we find that
    \begin{equation*}
        m_{\mathrm{lex}}^{(d)} \geq m^{(d)} = \max_{\ab \in U_d}
        \{m = m_1 \wedge \ldots \wedge m_{t_d} \mid m \text{ is a basis of } \bigwedge^{t_d}\ini_{\preceq}(I)_d\}.
    \end{equation*}
    Thus, if there exists $\ab \in U$ such that the initial ideal of $I$ is
    lexsegment, then $J$ is lexsegment.
\end{proof}

We compute the \groebner basis of a set of randomly chosen polynomials.
For example, let $\kk = \QQ, n = 3, s = 2$ and $d_1, d_2 = 2$.
In this case, $r_1 = r_2 = 6$ and $N=12$. We take a point randomly.
For example, we take
\begin{equation*}
    \ab = (8, -6, 9, -1, 1, 5, 1, 2, 7, -4, 5, -8) \in \Abb_{\kk}^N.
\end{equation*}
The point $\ab$ corresponds to an ideal
\begin{align*}
    I = (&8x_1^2 - 6x_1x_2 + 9x_1x_3 - x_2^2 + x_2x_3 + 5x_3^2, \\
    &x_1^2 + 2x_1x_2 + 7x_1x_3 - 4x_2^2 + 5x_2x_3 - 8x_3^2).
\end{align*}
Computing the reduced \groebner basis of $I$ with respect to the lexicographic
order, the initial ideal of $I$ is
\begin{equation*}
    \ini_{\preceq}(I) = (x_1^2, x_1x_2, x_1x_3^2, x_2^4).
\end{equation*}
It is the lexsegment ideal of the Hilbert series $\frac{t^2 + 2t + 1}{1 - t}$.
Furthermore, since the Hilbert series of $S/I$ is $\frac{t^2 + 2t + 1}{1 - t}$,
$I$ is a $U$-generic ideal.
Therefore, if $\kk = \QQ, n = 3, s = 2$ and $d_1, d_2 = 2$,
the initial ideal of $V$-generic ideals is the lexsegment ideal.
By performing similar computations in other cases, we obtain the following corollary.

\begin{corollary}\label{corollary:main corollary2}
    Assume the characteristic of $\kk$ is $0$.
    If $n,s$ and $d_i$'s satisfy one of the following conditions,
    then initial ideals of $V$-generic ideals with respect to the lexicographic order are lexsegment ideals.
    \begin{itemize}
        \item $n=3, s=2$, and $2 \leq d_1, d_2 \leq 3$.
        \item $n=3, s=3$, and $2 \leq d_1, d_2, d_3 \leq 3$.
        \item $n=4, s=3$, and $2 \leq d_1, d_2, d_3 \leq 3$.
    \end{itemize}
\end{corollary}

Now, we propose a new method of a computation of the initial ideals of $V$-generic ideals.
We use a stability condition of a \groebner basis.
This notion is used for computations of comprehensive \groebner systems.

Let $\tbar = \{a_1, \ldots, a_m\}$ be a set of parameter variables and
$\xbar = \{x_1, \ldots, x_n\}$ a set of main variables.
Let $\calG = \{G_1(\tbar, \xbar), \ldots, G_s(\tbar, \xbar)\} \subset \kk[\tbar][\xbar]$
be a set of polynomials and $I'$ an ideal in $\kk[\tbar][\xbar]$
generated by $\calG$. We fix a monomial order $\preceq$ on $\calM(\xbar)$.
We call $\calG$ a \textit{comprehensive \groebner basis} of $I'$ if,
for all $\ab \in \Abb_{\kk}^m$,
\begin{equation*}
    \sigma_{\ab}(\calG) = \{\sigma_{\ab}(G_i) = G_i(\ab, \xbar) \mid G_i \in \calG\}
\end{equation*}
is a \groebner basis of an ideal in $\kk[\xbar]$
\begin{equation*}
    \sigma_{\ab}(I') = \{\sigma_{\ab}(F) = F(\ab, \xbar) \mid F \in I'\}
\end{equation*}
with respect to $\preceq$. Also, we call $\{(A_1, \calG_1), \ldots, (A_l, \calG_l)\}$
a \textit{comprehensive \groebner system} of $I'$ if
$\Abb_{\kk}^m = A_1 \sqcup \ldots \sqcup A_l$ is an appropriate partition, and
for each $i = 1, \ldots, l$, the set of finite polynomials $\sigma_{\ab}(\calG_i)$
is a \groebner basis of
$\sigma_{\ab}(I')$ for all $\ab \in A_i$. Comprehensive \groebner basis and
comprehensive \groebner system were introduced by Weispfenning~\cite{WeispfenningCGB}.

If we apply an algorithm for computing comprehensive \groebner systems to
\begin{eqnarray*}
    F_1 &=& a_{1,1}x_1^{d_1} + a_{1,2}x_1^{d_1-1}x_2 + \cdots + a_{1,r_1}x_n^{d_1} \\
    &\vdots& \\
    F_s &=& a_{s,1}x_1^{d_s} + a_{s,2}x_1^{d_s-1}x_2 + \cdots + a_{s,r_s}x_n^{d_s}
\end{eqnarray*}
and compute a comprehensive \groebner system
$\{(A_1, \calG_1), \ldots, (A_l, \calG_l)\}$ of $I' = (F_1, \ldots, F_s)$,
we can determine whether there exists $\ab \in U$ such that $\ini_{\preceq}(I)$
is the lexsegment ideal. However, the computation of comprehensive \groebner system
of $I'$ took tremendous time so it is very difficult to compute it within
a realistic timeframe in general.
It is possible to determine the existence of such ideals without computing
a comprehensive \groebner basis.

There are several effective algorithms for computations of comprehensive \groebner
systems; for example, see~\cite{KapurSunWangCGS},~\cite{ManubensMontesCGS},
~\cite{NabeshimaCGS}, and~\cite{SuzukiSatoCGS}.
Some algorithms compute a \groebner basis of an ideal in $\kk[\tbar, \xbar]$
with respect to an inverse block monomial order $\preceq'$ on the first step
of algorithm.
The definition of an inverse block monomial order is the following.

\begin{definition}
    Let $\preceq_{\tbar}$ and $\preceq$ be monomial orders on $\calM(\tbar)$ and $\calM(\xbar)$,
    respectively.
    Let $u_1, u_2$ be monomials in $\calM(\tbar)$, and $v_1, v_2$ monomials
    in $\calM(\xbar)$. A monomial order $\preceq'$ on $\calM(\tbar \cup \xbar)$
    is defined as follows:
    \begin{equation*}
        u_1v_1 \preceq' u_2v_2 \overset{\mathrm{def}}{\iff}
        v_1 \preceq v_2 \text{ or } (v_1 = v_2, \text{ and } u_1 \preceq_{\tbar} u_2).
    \end{equation*}
    The monomial order $\preceq'$ is called an inverse block monomial order
    on $\calM(\tbar \cup \xbar)$.
\end{definition}

In what follows, we assume that a monomial order $\preceq'$ on $\calM(\tbar \cup \xbar)$
is an inverse block monomial order, and when we restrict $\preceq'$ to a
monomial order on $\calM(\xbar)$, it becomes the monomial order $\preceq$ we fixed,
which, in this section, is specifically the lexicographic order.
For a polynomial $F \in \kk[\tbar, \xbar]$, we sometimes consider $F$ as an element of
$\kk[\tbar][\xbar]$ and we define $\LC_{\preceq}(F) \in \kk[\tbar]$ as the leading
coefficient of $F$ with respect to restricted monomial order
$\preceq$ on $\calM(\xbar)$. Similarly, we define $\LM_{\preceq}(f) \in \kk[\xbar]$
as the leading monomial of $F$ with respect to $\preceq$.

The stability condition of Lemma~\ref{lemma:stability condition} is used
in~\cite{SuzukiSatoCGS}.

\begin{lemma}\label{lemma:stability condition}
    Let $I'$ be an ideal in $\kk[\tbar, \xbar]$ and
    $\calG = \{G_1, \ldots, G_r, G_{r+1}, \ldots, G_t\}$ a \groebner
    basis of $I'$ with respect to an inverse block monomial order $\preceq'$.
    For a point $\ab$, we assume $\sigma_{\ab}(\LC_{\preceq}(G_i)) \neq 0(1 \leq i \leq r)$
    and $\sigma_{\ab}(\LC_{\preceq}(G_i)) = 0(r+1 \leq i \leq s)$. Then,
    the following conditions are equivalent:
    \begin{enumerate}
        \item $\{\sigma_{\ab}(G_1), \ldots, \sigma_{\ab}(G_r)\}$ is a \groebner basis
        of $\sigma_{\ab}(I')$ with respect to the restricted monomial order $\preceq$;
        \item For all $i = r+1, \ldots, t$, $\sigma_{\ab}(G_i)$ is reducible to $0$
        modulo $\{\sigma_{\ab}(G_1), \ldots, \sigma_{\ab}(G_r)\}$.
    \end{enumerate}
\end{lemma}
\begin{proof}
    It follows from~\cite[Theorem 3.1]{KalkbrenerStability}.
\end{proof}

By applying Lemma~\ref{lemma:stability condition} to the polynomials,
that are the main focus of our paper, we obtain the following theorem.

\begin{theorem}\label{theorem:main theorem3}
    Let $I'$ be an ideal of $\kk[\tbar, \xbar]$ generated by
\begin{eqnarray*}
    F_1 &=& t_{1,1}x_1^{d_1} + t_{1,2}x_1^{d_1-1}x_2 + \cdots + t_{1,r_1}x_n^{d_1} \\
    &\vdots& \\
    F_s &=& t_{s,1}x_1^{d_s} + t_{s,2}x_1^{d_s-1}x_2 + \cdots + t_{s,r_s}x_n^{d_s}.
\end{eqnarray*}
Let $\calG$ be a \groebner basis with respect to an inverse block monomial order
$\preceq'$. Then, the monomial ideal in $\kk[\xbar]$ generated by
$\{\LM_{\preceq}(g) \mid g \in \calG\}$ coincides
with the initial ideal of $V$-generic ideals with respect to the restricted monomial
order $\preceq$.
\end{theorem}
\begin{proof}
    Let $W$ be the set of points in $\Abb_{\kk}^N$ that do not vanish on
    all of polynomials in $\{\LC_{\preceq}(g) \mid g \in \calG\}$.
    Since $\LC_{\preceq}(g)(\ab) \neq 0$ is equivalent to $\sigma_{\ab}(g) \neq 0$,
    the set of polynomials $\sigma_{\ab}(\calG)$ is a \groebner basis of $\sigma_{\ab}(I')$
    for all $\ab \in W$ by Lemma~\ref{lemma:stability condition}.
    Therefore, $\ini_{\preceq}(I)$ is the same for all $\ab \in W$.
    It is a monomial ideal generated by $\{\LM_{\preceq}(g) \mid g \in \calG\}$
    and it is equal to a monomial ideal generated by
    $\{\ini_{\preceq}(\sigma_{\ab}(g)) \mid g \in \calG\}$.
    Since we defined $V$ as the set of points that $\ini_{\preceq}(I)$ is the initial ideal of
    $V$-generic ideals as large as possible, $W \subset V$ or $W \cap V = \emptyset$.
    We assumed $\kk$ is infinite, so $W$ and $V$ are Zariski open dense sets in
    $\Abb_{\kk}^N$, we have $W \subset V$.
\end{proof}

By Theorem~\ref{theorem:main theorem2} and Theorem~\ref{theorem:main theorem3},
we can prove the non-existence of $\ab \in U$ such that $\ini_{\preceq}(I)$
is the lexsegment ideal.

\begin{corollary}\label{corollary:main corollary3}
    Assume the characteristic of $\kk$ is $0$.
    If $n = 4, s = 2, d_1=d_2=2$, there is no point $\ab \in U$ whose initial ideal
    with respect to the lexicographic order is the lexsegment ideal.
\end{corollary}
\begin{proof}
    We compute \groebner basis $\calG$ of two polynomials in $\QQ[\tbar, \xbar]$
    \begin{equation*}
        F_1 = a_{1,1}x_1^{2} + \cdots + a_{1,10}x_4^{2}, \quad F_2 = a_{2,1}x_1^{2} + \cdots + a_{2,10}x_4^{2}
    \end{equation*}
    with respect to the lexicographic order with
    \begin{equation*}
        x_1 \succeq \ldots \succeq x_4 \succeq t_{1,1} \succeq \ldots \succeq t_{1,10} \succeq t_{2,1} \succeq \ldots \succeq t_{2,10}.
    \end{equation*}
    Then, the monomial ideal generated by $\{\LM_{\preceq}(g) \mid g \in \calG\}$ is
    \begin{equation*}
        (x_1^2, x_1x_2, x_1x_3^2, x_2^4).
    \end{equation*}
    By Theorem~\ref{theorem:main theorem3}, this monomial ideal is the initial ideal of
    $V$-generic ideals with respect to the lexicographic order of $\xbar$.
    Furthermore, it is not the lexsegment ideal because $x_1x_3x_4^2 \succeq x_2^4$ but
    $x_1x_3x_4^2$ is not in the ideal.
    Thus, by the contraposition of Theorem~\ref{theorem:main theorem2},
    there is no point $\ab \in U$ whose initial ideal
    with respect to the lexicographic order is the lexsegment ideal.
\end{proof}

By Corollary~\ref{corollary:main corollary2} and Corollary~\ref{corollary:main corollary3},
we can propose these questions and expect that these statements are true.

\begin{question}
    If $n-s \geq 2$, is there no point $\ab \in U$ such that $\ini_{\preceq}(I)$
    is the lexsegment ideal for any $d_i$'s?
\end{question}

\begin{question}\label{question:lexicographic Moreno-Socias conjecture}
    If $n-s \leq 1$, is the initial ideal of $V$-generic ideals the lexsegment ideal
    for any $d_i$'s?
\end{question}

In particular, Question~\ref{question:lexicographic Moreno-Socias conjecture}
can be considered as lexicographic analogue of \moreno conjecture.

\bibliography{references}

\providecommand{\bysame}{\leavevmode\hbox to3em{\hrulefill}\thinspace}
\providecommand{\MR}{\relax\ifhmode\unskip\space\fi MR }
\providecommand{\MRhref}[2]{%
  \href{http://www.ams.org/mathscinet-getitem?mr=#1}{#2}
}
\providecommand{\href}[2]{#2}
\begin{thebibliography}{10}

\bibitem{provedFroebergConjecture1}
David~J. Anick, \emph{Thin algebras of embedding dimension three}, J. Algebra \textbf{100} (1986), no.~1, 235--259. \MR{839581}

\bibitem{BayerStillman}
David Bayer and Michael Stillman, \emph{A theorem on refining division orders by the reverse lexicographic order}, Duke Math. J. \textbf{55} (1987), no.~2, 321--328. \MR{894583}

\bibitem{BayerThesis}
David~Allen Bayer, \emph{T{HE} {DIVISION} {ALGORITHM} {AND} {THE} {HILBERT} {SCHEME}}, ProQuest LLC, Ann Arbor, MI, 1982, Thesis (Ph.D.)--Harvard University. \MR{2632095}

\bibitem{BeckerWeispfenning}
Thomas Becker and Volker Weispfenning, \emph{Gr\"obner bases}, Graduate Texts in Mathematics, vol. 141, Springer-Verlag, New York, 1993, A computational approach to commutative algebra, In cooperation with Heinz Kredel. \MR{1213453}

\bibitem{CapaverdeGaoMorenoSociasConjecture}
Juliane Capaverde and Shuhong Gao, \emph{Gr\"obner bases of generic ideals}, J. Algebra \textbf{641} (2024), 27--48. \MR{4672660}

\bibitem{RegularityofLexsegment}
Marc Chardin and Guillermo Moreno-Soc\'ias, \emph{Regularity of lex-segment ideals: some closed formulas and applications}, Proc. Amer. Math. Soc. \textbf{131} (2003), no.~4, 1093--1102. \MR{1948099}

\bibitem{CoxLittleOSheaIVA}
David~A. Cox, John Little, and Donal O'Shea, \emph{Ideals, varieties, and algorithms}, fourth ed., Undergraduate Texts in Mathematics, Springer, Cham, 2015, An introduction to computational algebraic geometry and commutative algebra. \MR{3330490}

\bibitem{DubeMcaulayConstant}
Thomas~W. Dub\'e, \emph{The structure of polynomial ideals and {G}r\"obner bases}, SIAM J. Comput. \textbf{19} (1990), no.~4, 750--775. \MR{1053942}

\bibitem{EisenbudText}
David Eisenbud, \emph{Commutative algebra}, Graduate Texts in Mathematics, vol. 150, Springer-Verlag, New York, 1995, With a view toward algebraic geometry. \MR{1322960}

\bibitem{FroebergConjecture}
Ralf Fr\"oberg, \emph{An inequality for {H}ilbert series of graded algebras}, Math. Scand. \textbf{56} (1985), no.~2, 117--144. \MR{813632}

\bibitem{FroebergGenericHilbertSeries}
Ralf Fr\"oberg and Clas L\"ofwall, \emph{On {H}ilbert series for commutative and noncommutative graded algebras}, J. Pure Appl. Algebra \textbf{76} (1991), no.~1, 33--38. \MR{1140638}

\bibitem{FroebergLoefwall}
\bysame, \emph{On {H}ilbert series for commutative and noncommutative graded algebras}, J. Pure Appl. Algebra \textbf{76} (1991), no.~1, 33--38. \MR{1140638}

\bibitem{provedFroebergConjecture3}
Ralf Fr\"oberg and Samuel Lundqvist, \emph{Questions and conjectures on extremal {H}ilbert series}, Rev. Un. Mat. Argentina \textbf{59} (2018), no.~2, 415--429. \MR{3900281}

\bibitem{Galligo}
Andr{\'e} Galligo, \emph{A propos du th{\'e}or{\`e}me de pr{\'e}paration de weierstrass}, Fonctions de Plusieurs Variables Complexes (Berlin, Heidelberg) (Fran{\c{c}}ois Norguet, ed.), Springer Berlin Heidelberg, 1974, pp.~543--579.

\bibitem{HashemiMacaulayConstant}
Amir Hashemi, Hossein Parnian, and Werner~M. Seiler, \emph{Computation of {M}acaulay constants and degree bounds for {G}r\"obner bases}, J. Symbolic Comput. \textbf{111} (2022), 44--60. \MR{4346711}

\bibitem{MonomialIdeals}
J\"urgen Herzog and Takayuki Hibi, \emph{Monomial ideals}, Graduate Texts in Mathematics, vol. 260, Springer-Verlag London, Ltd., London, 2011. \MR{2724673}

\bibitem{KalkbrenerStability}
Michael Kalkbrener, \emph{On the stability of {G}r\"obner bases under specializations}, J. Symbolic Comput. \textbf{24} (1997), no.~1, 51--58. \MR{1459670}

\bibitem{KapurSunWangCGS}
Deepak Kapur, Yao Sun, and Dingkang Wang, \emph{An efficient algorithm for computing a comprehensive {G}r\"obner system of a parametric polynomial system}, J. Symbolic Comput. \textbf{49} (2013), 27--44. \MR{2997838}

\bibitem{ManubensMontesCGS}
Montserrat Manubens and Antonio Montes, \emph{Improving the {DISPGB} algorithm using the discriminant ideal}, J. Symbolic Comput. \textbf{41} (2006), no.~11, 1245--1263. \MR{2267135}

\bibitem{MayrRitscherMacaulayConstant}
Ernst~W. Mayr and Stephan Ritscher, \emph{Dimension-dependent bounds for {G}r\"obner bases of polynomial ideals}, J. Symbolic Comput. \textbf{49} (2013), 78--94. \MR{2997841}

\bibitem{MillerSturmfels}
Ezra Miller and Bernd Sturmfels, \emph{Combinatorial commutative algebra}, Graduate Texts in Mathematics, vol. 227, Springer-Verlag, New York, 2005. \MR{2110098}

\bibitem{MorenoSociasConjecture}
Guillermo Moreno-Soc\'ias, \emph{Degrevlex {G}r\"obner bases of generic complete intersections}, J. Pure Appl. Algebra \textbf{180} (2003), no.~3, 263--283. \MR{1966660}

\bibitem{NabeshimaCGS}
Katsusuke Nabeshima, \emph{Generic {G}r\"obner basis of a parametric ideal and its application to a comprehensive {G}r\"obner system}, Appl. Algebra Engrg. Comm. Comput. \textbf{35} (2024), no.~1, 55--70. \MR{4685240}

\bibitem{PardueGenericSequence}
Keith Pardue, \emph{Generic sequences of polynomials}, J. Algebra \textbf{324} (2010), no.~4, 579--590. \MR{2651558}

\bibitem{provedFroebergConjecture2}
Richard~P. Stanley, \emph{Hilbert functions of graded algebras}, Advances in Math. \textbf{28} (1978), no.~1, 57--83. \MR{485835}

\bibitem{SuzukiSatoCGS}
Akira Suzuki and Yosuke Sato, \emph{A simple algorithm to compute comprehensive gröbner bases using gröbner bases}, Proceedings of the International Symposium on Symbolic and Algebraic Computation, ISSAC \textbf{2006} (2006), 326--331.

\bibitem{WeispfenningCGB}
Volker Weispfenning, \emph{Comprehensive {G}r\"obner bases}, J. Symbolic Comput. \textbf{14} (1992), no.~1, 1--29. \MR{1177987}

\end{thebibliography}
\bibliographystyle{amsplain}
\end{document}